\documentclass[11pt,a4paper]{article}
\usepackage{a4wide}
\RequirePackage{amsthm}

\usepackage{makeidx}         
\usepackage{graphicx}        
\usepackage{multicol}        
\usepackage[varvw]{newtxmath}       

\usepackage{amsfonts}

\usepackage{pgfplots}
\usepackage{hyperref}
\usepackage{environ}
\pgfplotsset{compat=newest}
\pgfplotsset{plot coordinates/math parser=false}

\usepackage[nameinlink]{cleveref}
\usepackage{mathtools}
\usepackage{mathrsfs}

\usetikzlibrary{arrows.meta}
\newcommand{\R}{\mathbb{R}}

\renewcommand{\epsilon}{\varepsilon}
\newcommand{\N}{\mathbb N}
\newcommand{\Z}{\mathbb Z}

\newcommand{\sL}{{\mathsf{L}\!}}

\newcommand{\id}{\ensuremath{\mathrm{Id}}}
\newcommand{\loc}{\textnormal{loc}}

\newcommand{\dd}{\ensuremath{\,\mathrm{d}}}
\DeclareMathOperator{\supp}{supp}
\DeclareMathOperator*{\esssup}{ess-\sup}
\DeclareMathOperator*{\essinf}{ess-\inf}
\newcommand{\ndt}{{\eta}}
\newcommand{\wt}{\gamma}

\newcommand{\Ne}{{N_\ndt}}

\newcommand{\norm}[1]{{\left\|#1\right\|}}

\newcommand{\Dx}{\Delta x}
\newcommand{\Dt}{\Delta t}
\newcommand{\jmh}{{j-\frac{1}{2}}}
\newcommand{\jph}{{j+\frac{1}{2}}}
\newcommand{\jmo}{{j-1}}

\theoremstyle{plain}						
\newtheorem{theorem}{Theorem}[section]

\theoremstyle{definition}

\newtheorem{assumption}[theorem]{Assumption}
\theoremstyle{remark}
\newtheorem{remark}[theorem]{Remark}

\newlength\fwidth

\makeindex             

\allowdisplaybreaks[4]
\title{Conservation laws with nonlocality in density and velocity and their applicability in traffic flow modelling}
  
\author{Jan Friedrich\footnotemark[1],\ Simone G\"ottlich\footnotemark[2],\ Alexander Keimer\footnotemark[3]\ \ and Lukas Pflug\footnotemark[4]}

\begin{document}

\footnotetext[1]{RWTH Aachen University, Institute of Applied Mathematics, 52064 Aachen, Germany ({friedrich@igpm.rwth-aachen.de})}
\footnotetext[2]{University of Mannheim, Department of Mathematics, 68131 Mannheim, Germany ({goettlich@uni-mannheim.de})}
\footnotetext[3]{Friedrich-Alexander Universität Erlangen-Nürnberg, Department Mathematik, Cauerstr. 11, 91058 Erlangen, Germany ({alexander.keimer@fau.de})}
\footnotetext[4]{Friedrich-Alexander Universität Erlangen-Nürnberg, Competence Unit for Scientific Computing, Martensstr. 5a, 91058 Erlangen, Germany ({lukas.pflug@fau.de})}
%
%
\maketitle


\begin{abstract}
In this work we present a nonlocal conservation law with a velocity depending on an integral term over a part of the space. 
The model class covers already existing models in literature, but it is also able to describe new dynamics mainly arising in the context of traffic flow modelling. We prove the existence and uniqueness of weak solutions of the nonlocal conservation law. Further, we provide a suitable numerical discretization and present numerical examples. 
\end{abstract}
 

\section{Modelling equations}
In recent years, the mathematical analysis on nonlocal conservation laws \cite{aggarwal,teixeira,chiarello,pflug} has drawn increased attention.
These nonlocal models are capable to describe a variety of applications such as traffic flow modelling 
\cite{bayen2022modeling,blandin2016well,chiarello2020micro,chiarello2019non-local,friedrich2018godunov,keimer1,kloeden,lee2019thresholds}, 
supply chains \cite{wang,keimer2,keimer3}, 
sedimentation processes \cite{betancourt,burger2022hilliges}, pedestrian dynamics \cite{colombo2012}, particle growth \cite{pflug2020emom,rossi2020well}, crowd dynamics and population modelling 
\cite{colombo_nonlocal,lorenz2019nonlocal} 
as well as opinion formation \cite{spinola,piccoli2018sparse}.

By a nonlocal conservation law we mean a conservation law, in which the velocity depends on a space dependent integral term of some quantity of interest.
In the literature, there are two main approaches: 
on the one hand the averaging is done over the density which is then used to determine the velocity, e.g. \cite{betancourt,blandin2016well,chiarello}; 
on the other hand, the averaging is done directly over the velocity \cite{friedrich2022conservation,friedrich2018godunov,burger2022hilliges}.
Both approaches can be found in applications for traffic flow and sedimentation.
Here, we want to consider a unifying approach, incorporating both averaging into one equation. Consider that \(V_{1},V_{2}:\R\rightarrow\R\) are given velocities, we study the nonlocal conservation law
\begin{equation}
\begin{aligned}
\partial_{t}q+ \partial_{x}\big(V_1\big(\gamma\ast V_2(q)\big)q\big)&=0 &&(t,x)\in(0,T)\times\R\\
q(0,x)&=q_{0}(x)&& x\in\R
\end{aligned}
\label{eq:NLmixed}
\end{equation}
with 
\begin{align*}
    \gamma \ast V_2(q)(t,x)\coloneqq\int_x^\infty \wt(y-x)V_2(q(t,y))\dd y.
\end{align*}
with \(q_{0}\) denoting the initial density and \(\gamma\) a weight function appearing in the nonlocal term.
As previously stated, the equation entails  also the pure nonlocality in the velocity (\(V_{1}\equiv \id\)) and the pure nonlocality in the density (\(V_{2}\equiv\id\)).

Let us explain the idea behind the proposed model \eqref{eq:NLmixed} for traffic flow in greater detail: Drivers adapt their velocity according to a mean of some quantity of interest.
This quantity is computed out of the density by the function $V_2$.
Then, the function $V_1$ transforms the quantity of interest into a velocity.
In particular, for $V_2\equiv \id$ the quantity of interest is the density and $V_1$ a suitable velocity function such that we obtain the averaging over the density, see \cite{blandin2016well}.
If $V_2$ transforms the density directly into a velocity, we choose $V_1\equiv \id$ and obtain the averaging over the velocity as proposed in \cite{friedrich2018godunov}.
But besides these two choices the model \eqref{eq:NLmixed} is capable to describe much more effects.
We will present examples in \cref{sec:numericexample}.

\section{Existence and uniqueness of solutions}
In this section, we shortly discuss the well-posedness, i.e.\ the existence and uniqueness of solutions. We will also dwell on the maximum principle, a quite important property as it tells that the time dependent density is always bounded between a minimal and maximal density which stems from the initial density on the road. For traffic applications this property is crucial to have.

We start with an existence and uniqueness result on small time horizon and require for now the following assumptions:
\begin{assumption}[Minimal assumptions on the involved input data]\label{ass:1} We assume that
\[\bullet\ q_{0}\in L^{\infty}(\R;\R_{\geq 0})\cap BV(\R),\ 
\bullet\ V_{1},V_{2}\in W^{1,\infty}_{\text{loc}}(\R),\
\bullet\ \gamma\in BV(\R_{> 0};\R_{\geq 0})
\]
with \(TV(\R)\coloneqq \big\{f\in L^{1}_{\loc}(\R):\ |f|_{TV(\R)}<\infty\big\}\) and
\[
\|f\|_{BV(\R)}\coloneqq \|f\|_{L^{1}(\R)}+ \supp_{\phi\in C^{\infty}_{\text{c}}(\R):\ \|\phi\|_{L^{\infty}(\R)}\leq 1}\int_{\R}\phi'(x)f(x)\dd x,\ f\in L^{1}_{\loc}(\R).
\]
\end{assumption}
As we will talk about weak solutions, we just mention that we mean with weak solutions the canonical definition (see for instance \cite[Definiton 2.13]{pflug}).

\begin{theorem}[Existence and uniqueness on a small time horizon]
Let assumption \ref{ass:1} hold. Then, there exists \(T\in\R_{>0}\) so that \eqref{eq:NLmixed} admits a unique weak solution \[
q\in C\big([0,T];\sL^{1}(\R)\big)\cap L^{\infty}\big((0,T);BV(\R)\big).\]
\end{theorem}
\begin{proof}
    The proof consists of applying a fixed-point approach on the nonlocal term and the solution as follows.
    Assume that we know the solution to the balance law \(q\), we can plug it into the nonlocal operator and have
    \begin{align*}
        \gamma\ast V_{2}(q)(t,x)=\int_{x}^{\infty}\gamma(y-x)V_{2}(q(t,y))\dd y
    \end{align*}
    from which we can compute the ``velocity'' of the conservation law as
    \begin{equation}
(t,x)\mapsto V_{1}\big(\gamma\ast V_{2}(q)(t,x)\big),\ (t,x)\in (0,T)\times\R.\label{eq:velocity_nonlocal}
    \end{equation}
    Assuming that \(q\) has the postulated regularity \(C\big([0,T];\sL^{1}(\R)\big)\cap L^{\infty}\big((0,T);BV(\R)\big)\), we can observe that the velocity is Lipschitz-continuous. We can thus invoke characteristics and state the solution once more as
    \begin{equation}
q(t,x)=q_{0}(\xi_{q}(t,x;0))\partial_{2}\xi_{q}(t,x;0),\label{eq:42}
    \end{equation}
    where \(\xi_{q}\) is the solution to the characteristics, i.e.\ 
    \[
\xi(t,x;\tau)=x+\int_{t}^{\tau} V_{1}\big(\gamma\ast V_{2}(q)(s,\xi(t,x;s))\big)\dd s,\quad (t,x,\tau)\in (0,T)\times\R\times(0,T).
    \]
As \(\xi_{q}\) depends on the solution \(q\), the identity in \cref{eq:42} is actually a fixed-point problem in the solution \(q\). By means of Banach's fixed-point theory we can then prove that a unique solution to \cref{eq:42} exists. The uniqueness carries over from the fixed-point argument. As can be seen, the choice of topology is crucial to obtain a unique solution and the small time horizon guarantees that the fixed-point mapping is a contraction (involving the properties of the characteristics). We do not detail this further, but just mention that these approaches have been used in several publications \cite{bayen2022modeling}, often not in the solution but in the nonlocal term \cite{pflug}.
\end{proof}
Having established the well-posedness of solutions, we can restrict the velocities \(V_{1},V_{2}\) in a reasonable way -- from an application point of view -- so that a maximum principle holds which itself implies the existence and uniqueness of solutions on any arbitrary time horizon.
\begin{theorem}[Maximum principle/existence of solutions on finite time horizon]
Let assumption \ref{ass:1} hold and assume in addition that
\begin{equation}
\begin{gathered}
\big(V_{1}'\leqq0\wedge\ V_{2}'\geqq0\big) \vee \big(V_{1}'\geqq0\wedge\ V_{2}'\leqq0\big)\ \text{ on }\ \Big[\essinf_{x\in\R}q_{0}(x),\esssup_{x\in\R}q_{0}(x)\Big]\\
\gamma \ \text{ monotonically decreasing}.
\end{gathered}
\label{eq:sign_V_1_V_2}
\end{equation}
Then, for every \(T\in\R_{>0}\) there is a unique weak solution \(q\in C\big([0,T];L^{1}(\R)\big)\cap L^{\infty}((0,T);BV(\R))\) of the nonlocal conservation law in \cref{eq:NLmixed}, and the following maximum principle holds:
\begin{align}
\essinf_{x\in\R}q_{0}(x)\leq q(t,x)\leq \esssup_{x\in\R}q_{0}(x)\qquad (t,x)\in (0,T)\times\R \text{ a.e.}.\label{eq:maximum_principle}
\end{align}
\end{theorem}
\begin{proof}
We only sketch the proof and assume that the solution to the conservation law with nonlocality as in \cref{eq:NLmixed} is smooth and compactly supported (this can be obtained by a classical approximation argument, we do not detail here). Then, we can assume that we are at a spatial location \(\tilde{x}\in\R\), where the maximum is attained.
Recalling the differential equation which is due to the higher regularity now satisfied in its strong version, we obtain for the time derivative at \(\tilde{x}\) and any given time \(t\in[0,T]\)
\begin{align}
\partial_{t}q(t,\tilde{x})&=-V_{1}'(\gamma\ast V_{2}(q(t,\cdot))(\tilde{x}))\partial_{2}\gamma\ast V_{2}(q(t,\cdot))(\tilde{x}) q(t,\tilde{x})\label{eq:partial_t_q_1}\\
&\quad -\big(V_1\big(\gamma\ast V_{2}(q(t,\cdot))(\tilde{x})\big)\partial_{2}q(t,\tilde{x})\big).\label{eq:partial_t_q_2}
\intertext{Recalling that \(\tilde{x}\) was one of the maximal points, we know that \(\partial_{2}q(t,\tilde{x})=0\) and get}
&=-V_{1}'(\gamma\ast V_{2}(q(t,\cdot))(\tilde{x}))\partial_{2}\gamma\ast V_{2}(q(t,\cdot))(\tilde{x}) \underbrace{q(t,\tilde{x})}_{\geq 0}.\label{eq:124}
\end{align}
Let us distinguish two cases:
\begin{itemize}
\item Assume that \(V_{1}\) is monotonically increasing (i.e.,\ \(V_{1}'\geqq 0\)) and \(V_{2}\) monotonically decreasing (i.e., \ \(V_{2}'\leqq 0\)) as well as \(\gamma\) monotonically decreasing, we can continue the computations and have
\begin{align*}
    \cref{eq:124}&=\underbrace{-V_{1}'(\gamma\ast V_{2}(q(t,\cdot))(\tilde{x}))}_{\leq 0}\underbrace{\partial_{2}\gamma\ast V_{2}(q(t,\cdot))(\tilde{x})}_{\geq0} \underbrace{q(t,\tilde{x})}_{\geq 0}
\end{align*}
Thus, the time derivative at the maximal points is negative, resulting in the mentioned upper bounds on the solution.
\item Assume that \(V_{1}\) is monotonically decreasing (i.e.,\ \(V_{1}'\geqq 0\)) and \(V_{2}\) monotonically increasing (i.e., \ \(V_{2}'\leqq 0\)) we can argue similarly to obtain again the upper bound on the solution.
\end{itemize}
For the lower bound, one can consider the time derivative at a minimal point and obtain that the derivative is then nonnegative. Finally, having uniform bounds as in \cref{eq:maximum_principle}, the solution can be extended to any finite time horizon as the velocity in \cref{eq:velocity_nonlocal} will remain uniformly Lipschitz-continuous with a Lipschitz-constant independent of the time considered. Thus, the typical clustering in time argument can be applied leading to the existence and uniqueness of weak solutions on any finite time horizon.
\end{proof}
\begin{remark}[Reasonability for the sign restriction in \cref{eq:sign_V_1_V_2} and \(\gamma\)]
The sign restrictions on the velocities are quite reasonable as one of them should be a decreasing function with regard to the traffic density while, when both of them would be decreasing or increasing, the composed velocity as in \cref{eq:velocity_nonlocal} would not be decreasing with regard to the traffic density, something nonphysical from the point of view of traffic.
The assumption on the nonlocal kernel acting at the position \(x\in\R\) from \(x\) to possibly \(\infty\) and being according to \cref{eq:sign_V_1_V_2} monotonically increasing can be understood from a traffic's perspective as follows: The velocity of the current density is only adjusted based on what is ahead in traffic, and the further away the traffic is, the less impact (or even no impact in the case that \(\gamma\) is compactly supported) the traffic information will have. This is in line with previous assumptions on the kernel having been explored in \cite{aggarwal,teixeira,bayen2022modeling,blandin2016well,chiarello,friedrich2018godunov,pflug, keimer1}.
\end{remark}
\begin{remark}[Monotonicity preserving dynamics under more  restrictive velocity \(V_{1}\)] 
Another interesting fact worth mentioning is the monotonicity  preserving dynamics, provided that \cref{eq:sign_V_1_V_2} holds and in addition \(V_{1}''\leq 0\), thus necessitating \(V_{1}\in W^{2,\infty}_{\loc}(\R)\). Consider again smooth solutions and assume for now (only for reasons of simplicity) that we have a piece-wise constant kernel, i.e.\ 
        \[
        \gamma\ast V_{2}(q)(t,x)=\tfrac{1}{\eta}\int_{x}^{x+\eta}V_{2}(q(t,y))\dd y,\ (t,x)\in[0,T]\times\R
\]
(the estimate can of course be made for general monotonically decreasing kernels of sufficient regularity).
Then, recall \crefrange{eq:partial_t_q_1}{eq:partial_t_q_2} and assume
 that the initial datum is monotonically increasing and that we are at a time \(t\in [0,T]\) so that for the first time there exists \(\tilde{x}\in\R\) with \(\partial_{x}q(t,x)\big|_{x=\tilde{x}}=0\), i.e., a point where the monotonicity might break. Then, we manipulate terms as follows:
\begin{align*}
-\partial_{t,x}q(t,x)&=\tfrac{1}{\eta^{2}}V_{1}''(\ldots) \big(V_{2}(q(t,x+\eta))-q(t,x)\big)^{2}q(t,x)\\
&\quad +\tfrac{1}{\eta}V_{1}'(...) \Big(V_{2}'(q(t,x+\eta))q_{x}(x+\eta)-V_{2}'(q(t,x))q_{x}(x)\big)q(t,x)\\
&\quad + \tfrac{1}{\eta}2V_{1}'(...) \big(V_{2}(q(t,x+\eta))-V_{2}(q(t,x)))q_{x}(t,x)\\
&\quad+V_1\big(\ldots\big)q_{xx}(t,x)
\intertext{and evaluate now at \(x=\tilde{x}\) so that \(\partial_{x}q(t,x)\geq 0\ \forall x\in\R,\ \partial_{x}q(t,x)\big|_{x=\tilde{x}}=0,\ \partial_{x}^{2}q(t,x)\big|_{x=\tilde{x}}=0\) (thus assuming \(\partial_{x}q)\) is minimal at \(x=\tilde{x}\))}
\partial_{2}q_{t}(t,\tilde{x})&=-\tfrac{1}{\eta^{2}}V_{1}''(...) \big(V_{2}(q(t,\tilde{x}+\eta))-V_{2}(q(t,\tilde{x}))\big)^{2}q(t,\tilde{x})\\
&\quad -\tfrac{1}{\eta}V_{1}'(...) \big(V_{2}'(q(t,\tilde{x}+\eta))\partial_2 q(t,\tilde{x}+\eta)-V_{2}'(q(t,\tilde{x}))\partial_{2}q(t,\tilde{x})\big)q(t,\tilde{x})\\
&\quad- \tfrac{1}{\eta}2V_{1}'(...) (V_{2}(q(t,\tilde{x}+\eta))-V_{2}(q(t,\tilde{x})))\partial_{2}q(t,\tilde{x})\\
\intertext{using that \(V_1''\leq 0\) and that \(\partial_{x}q(t,x)\big|_{x=\tilde{x}}=0\) as well as the sign restrictions on \(V_{1}',V_{2}'\)}
& =\underbrace{-V_{1}'(...) V_{2}'(q(t,\tilde{x}+\eta))}_{\geq 0}\partial_2 q(t,\tilde{x}+\eta)\underbrace{q(t,\tilde{x})}_{\geq 0}\\
&\geq \underbrace{-V_{1}'(...) V_{2}'(q(t,\tilde{x}+\eta))}_{\geq 0}\partial_2 q(t,\tilde{x})\underbrace{q(t,\tilde{x})}_{\geq 0}=0.
\end{align*}
Thus, the solution remains monotonically increasing for all times.
For monotonically decreasing initial datum, we would require \(V_{1}''\geqq0\), and could then establish similarly that the solutions then remains decreasing. This assumption is somewhat in line with observations in \cite[Theorem 4.13 \& Theorem 4.18]{pflug4}, where exactly the same assumption can be found for the simpler case \(V_{2}\equiv \id\). Choosing \(V_{1}\equiv \id,\) we obtain on the other hand no restrictions on \(V_{2}\) to preserve monotonicity which is in line with \cite[Theorem 5.1]{friedrich2022conservation}.
\label{rem:monotonicity}
\end{remark}

\section{Numerical discretization}\label{sec:numericscheme}
To construct an approximate solution we adapt the approaches presented in \cite{friedrich2018godunov,friedrich2023numerical} to the setting in \cref{eq:NLmixed}.
In particular, we rely on a Godunov-type scheme.
Hence, we discretize space and time by an equidistant grid with the step sizes $\Dx\in\R_{>0}$ in space and $\Dt\in\R_{>0}$ in time, such that $t^n=n\Dt$ with $n\in\N$ describes the time mesh and $x_j=j\Dx,\ j\in\Z$ the cell centres of the space mesh with the cell interfaces $x_\jmh$ and $x_\jph$.
The finite volume approximation $q^{\Dx}$ is given by $q^{\Dx}(t,x)=q_j^n$ for $(t,x)\in[t^n,t^{n+1})\times \big[x_\jmh,x_\jph\big)$ and we approximate the initial data by
\begin{align}\label{eq:ininumeric}
    q_j^0=\tfrac{1}{\Dx} \int_{x_\jmh}^{x_\jph} q_0(x) dx,\quad j\in\Z.
\end{align}
Following \cite{friedrich2018godunov} the scheme is given by 
\begin{align}\label{eq:scheme}
    q_j^{n+1}=q_j^n-\lambda \left(q_j^nV_j^n-q_{j-1}^nV_{j-1}^n\right)\quad\text{with}\quad \lambda\coloneqq\tfrac{\Dt}{\Dx},
\end{align}
with the nonlocal term $V_j^n$.
To compute this term numerically we need to restrict the support of the kernel $\gamma$ on an interval $[0,\eta]$ with $\eta>0$.
For a spatial step size of $\Delta x$, $\eta$ must be chosen such that $\int_{\eta}^\infty \gamma\big(y\big)\dd y=\mathcal{O}(\Delta x)$ to maintain first order accuracy.
If the support of the kernel is already compact, we simply choose $\eta$ as the supremum of the support.
Then, the nonlocal term is approximated for $\Ne\coloneqq\lfloor \ndt/\Dx\rfloor$ by
\begin{equation}\label{eq:approximatedNLterm}
V_j^n\coloneqq V_1\left(\sum_{k=0}^{\Ne-1} \gamma_k V_2(q_{j+k+1}^n)\right)\quad\text{with}\quad 
    \gamma_k=\int_{k\Dx}^{(k+1)\Dx} \wt (x) dx.
\end{equation}
The weights $\gamma_k$ need to be computed exactly and the CFL condition is given by
\begin{align}\label{eq:CFL}
    \lambda\leq \frac{1}{\gamma_0 \norm{V_1'}_{L^\infty(V_2)}\norm{V_2'}_{L^\infty(\id)}\esssup_{x\in\R}q_{0}(x)+\norm{V_1}_{L^\infty(V_2)}},
\end{align}
where we use the notation 
\[\norm{\cdot}_{L^\infty(f)}\coloneqq\norm{\cdot}_{L^\infty((\essinf_{x\in\R}f(q_{0}(x)),\esssup_{x\in\R}f(q_{0}(x))))}\]
for simplicity.
We now prove that the numerical discretization fulfills the same maximum principle as the analytical solution.
\begin{theorem}[Discrete version of the maximum principle]
Given assumption \ref{ass:1} and the conditions in \cref{eq:sign_V_1_V_2}, for a given initial datum $q_j^0,\ j\in \Z$ as in \cref{eq:ininumeric} with $q_m=\min_{j\in\Z} q_j^0$ and $q_M=\max_{j\in\Z} q_j^0$, the scheme \eqref{eq:scheme}--\eqref{eq:approximatedNLterm} fulfills under the CFL condition \eqref{eq:CFL}
\[q_m\leq q_j^n\leq q_M,\ j\in\Z,\ n\in\N.\]
\end{theorem}
\begin{proof}
Following closely the proof of \cite[Theorem 3.1]{friedrich2018godunov} it turns out that we only need to consider the difference between two nonlocal velocities.
If they satisfy
\begin{align*}
     V_{\jmo}^{n}-V_{j}^{n}\leq \norm{V_1'}_{L^\infty(V_2)}\norm{V_2'}_{L^\infty(\id)} \gamma_0 (\rho_M-\rho_j^n),
\end{align*}
then we obtain $q_j^{n+1}\leq q_{M}$ by following the corresponding steps in \cite[Theorem 3.1]{friedrich2018godunov}, using the CFL condition \eqref{eq:CFL} and the scheme \eqref{eq:scheme}.
For $j\in\Z$ and a fixed $n\in\N$ with $n>0$ we obtain 
\begin{align*}
    V_{\jmo}^{n}-V_{j}^{n}&= 
    V_1\left(\sum_{k=0}^{\Ne-1} \gamma_k V_2(q^n_{j+k})\right)-V_1\left(\sum_{k=0}^{\Ne-1} \gamma_k V_2(q^n_{j+k+1})\right)\\
    &=V_1'(\xi_j)\left(\sum_{k=1}^{\Ne-1}(\gamma_k-\gamma_{k-1})V_2(q_{j+k}^n)-\gamma_{{\Ne-1}} V_2(q_{j+\Ne}^n)+\gamma_0 V_2(q_j^n)\right)\\
    \intertext{with $\xi_j\in\R$ appropriatly chosen (it exists thanks to the mean value theorem),}
    &\leq V_1'(\xi_j)\left(\sum_{k=1}^{\Ne-1}(\gamma_k-\gamma_{k-1})V_2(q_M)-\gamma_{{\Ne-1}}V_2(q_{M})+\gamma_0 V_2(q_j^n)\right)\\
    \intertext{which holds since $\gamma_k\leq \gamma_{k-1}$ for $k=1,\dots,\Ne-1$ and because of the signs of $V_1'$ and $V_2'$ in \eqref{eq:sign_V_1_V_2},}
    &=V_1'(\xi_j)\gamma_{0}\left(V_2(q_j^n)-V_2(q_{M})\right)\leq \norm{V_1'}_{L^\infty(V_2)}\norm{V_2'}_{L^\infty(\id)} \gamma_0 (q_M-q_j^n),
\end{align*}
where we use again the signs of $V_1'$ and $V_2'$ given by \cref{eq:sign_V_1_V_2}. 
Analogously, we can prove
\begin{align*}
     V_{\jmo}^{n}-V_{j}^{n}\geq \norm{V_1'}_{L^\infty(V_2)}\norm{V_2'}_{L^\infty(\id)} \gamma_0 (q_m-q_j^n)
\end{align*}
which can be used to prove the lower bound.
\end{proof}
We note that following similar steps as in \cite{friedrich2018godunov}, it is possible to derive bounded variation estimates on the approximate solution such that the convergence of the scheme against a weak solution can be obtained.
We do not go into details here.

\section{Numerical examples}\label{sec:numericexample}
Let us present two examples which can be described by \cref{eq:NLmixed} and give a  suitable interpretation from a traffic modelling point of view.
During this section we set $\gamma(x)=2\frac{\ndt-x}{\ndt^2}\chi_{[0,\ndt]}$ with $\ndt=0.5$ and
the initial data to \(q_{0}\equiv\tfrac{1}{4}+\tfrac{1}{2}\chi_{[-0.5,0.5]}\). We are interested in the approximate solution at the time $t=0.5$ for a spatial discretization given by $\Dx=10^{-3}$ and a time grid size of $\Dt=\Dx/(3\gamma_0+1)$. We note that the CFL condition is slightly stricter than the one given by \eqref{eq:CFL}, but it allows to choose the same CFL condition in all simulations.\\
In the first example, drivers might not be able to perceive the true density on a road.
They only estimate the observed density ahead of them and base their velocity on this estimation.
Hence, $V_2$ expresses the estimated density in dependence of the true density.
A possible choice is $V_2(q)=q+\varepsilon q(1-q)$ for $\varepsilon\in[-1,1]\setminus \{0\}$, i.e. underestimation of the density for $\varepsilon<0$ and overestimation for $\varepsilon>0$.
Further, if $q_0\in[0,1]$, it follows $V_2'\geq 0$. Here,
$V_1$ is the velocity function to determine the velocity out of the estimated density with $V_1'\leq 0$, e.g.\ $V_1(q)=1-q^2$.
\Cref{fig:estimation} shows the approximate solutions for different values of $\varepsilon$.
In particular, the case $\varepsilon=0$ is the nonlocal in density model, in which the drivers have perfect knowledge of the density.
It can be seen that an underestimation of the velocity results in a higher density while for overestimation the density is lower.
Further, in the case of underestimation the density is located further downstream.
\begin{figure}
    \centering
    \setlength{\fwidth}{0.8\textwidth}
   \input{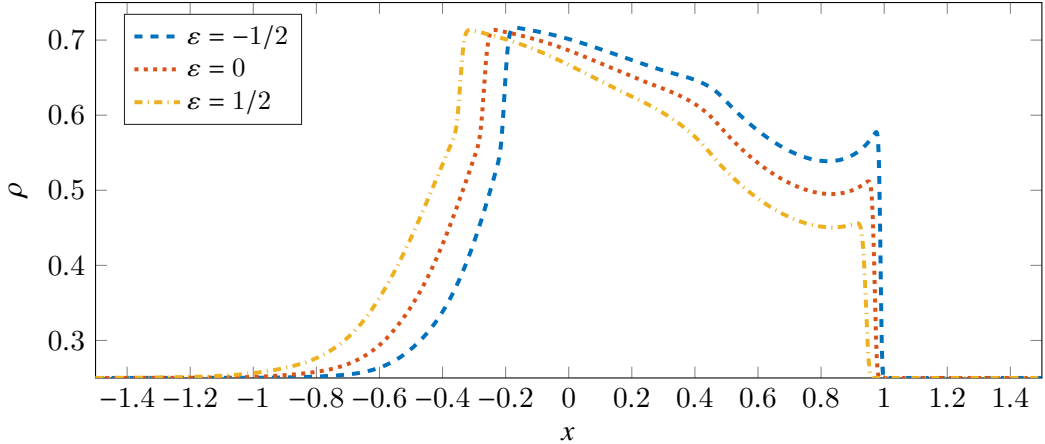}
    \caption{Approximate solution of \eqref{eq:NLmixed} at $t=0.5$ with  $V_2(q)=q+\varepsilon q(1-q)$ and $V_1(q)=1-q^2$.}
    \label{fig:estimation}
\end{figure}

Even more interesting is the case that drivers consider a mixture of relative velocity and relative density for their movement.
Hence, for a suitable velocity function $v$ (meaning $v'\leq 0$) to estimate the velocity and $\alpha\in[0,1]$, which expresses the preference to adapt more according to the density or the velocity, a quantity of interest might be
\begin{align}\label{eq:preferences}
    V_2(q)=\alpha \tfrac{q}{q_{\max}}+(1-\alpha)\left(1-\tfrac{v(q)}{v_{\max}}\right).
\end{align}
Again $V_1$ transforms this quantity into a velocity.
\Cref{fig:motivation} displays the approximate solutions for different values of $\alpha$ and $V_2(q)$ as in \eqref{eq:preferences} with $v(q)=1-q^2$, $q_{\max}=v_{\max}=1$ and $V_1(q)=(1-q)^2$.
Interestingly, we can see that if more preference is given to the velocity, the front end of the traffic jam moves faster, the peak of the traffic jam decreases, but at the back end of the traffic jam the density is higher.
\begin{figure}
    \centering
    \setlength{\fwidth}{0.8\textwidth}
   \input{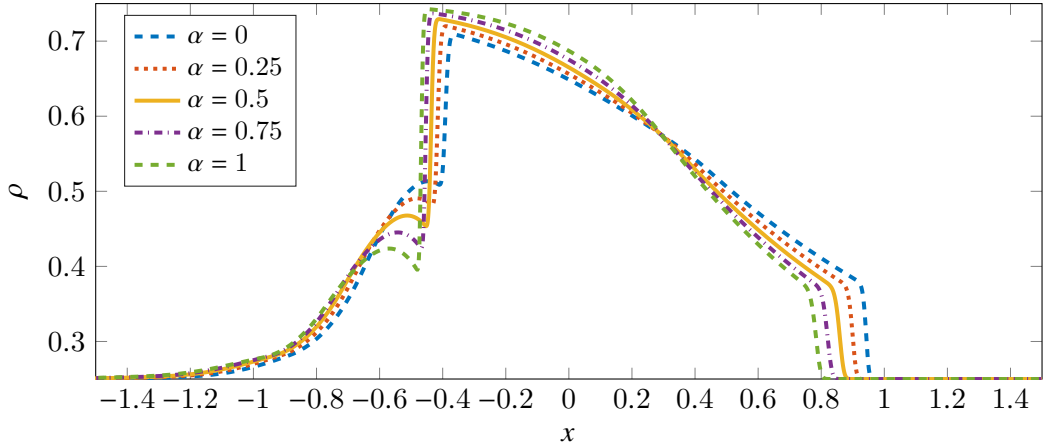}
    \caption{Approximate solution of \eqref{eq:NLmixed} at $t=0.5$ with $V_2$ as in \eqref{eq:preferences} and $V_1(q)=(1-q)^2$.}
    \label{fig:motivation}
\end{figure}
\begin{remark}
    We note that the monotonicity property in \cref{rem:monotonicity} can be seen in \cref{fig:estimation} and \cref{fig:motivation}.
    In particular, in the first example $V_1''<0$ holds and hence the monotonicity is only kept for the increasing part, see \cref{fig:estimation}, and in the second example due to $V_1''>0$ the montonicity is kept on the decreasing part.
\end{remark}

\section{Conclusion}
We have presented results on the well-posedness of a nonlocal conservation law incorporating both approaches of space averaging, i.e., mean density and mean velocity. Numerical examples demonstrate the performance of the model.\\
Future work may include studying the behavior of the model when the kernel function tends to the Dirac delta. This question has been intensively studied in the literature for the case $V_2\equiv \id$ and $V_1\equiv \id$, see e.g. \cite{teixeira,bressan2021entropy, colombo2023nonlocal,friedrich2022conservation,pflug4} and the references therein, but so far not for the general case of eq. \eqref{eq:NLmixed}.

\section*{Acknowledgement}
Lukas Pflug thanks for the support by the Collaborative Research Centre 1411 “Design of Particulate Products” (Project-ID 416229255).
Jan Friedrich was supported by the German Research Foundation (DFG) under grant HE 5386/18-1, 19-2, 22-1, 23-1 and
Simone G\"ottlich under grant GO 1920/10-1.

\bibliographystyle{plain}
\bibliography{references}

\end{document}